\documentclass[12pt]{amsart}

\usepackage{amsmath}
\usepackage{amsfonts}
\usepackage{amssymb}
\usepackage{amsthm}
\usepackage{hyperref}
\usepackage{enumerate}
\usepackage{cleveref}
\usepackage{mathrsfs}
\usepackage[top=0.9in, bottom=1.15in, left=1.1in, right=1.1in]{geometry}
 \usepackage{hyperref}
\hypersetup{colorlinks=true,linkcolor=blue,citecolor=magenta}
\newtheorem{theorem}{Theorem}

\newtheorem{definition}[theorem]{Definition}
\newtheorem*{remark}{Remark}

\Crefname{conjecture}{Conjecture}{Conjectures}

\theoremstyle{plain}

\theoremstyle{plain}

\author{A. P. Akande and Robert Schneider}

\address{Department of Mathematics\newline
University of Georgia\newline
Athens, Georgia 30602, U.S.A.}
\email{agbolade.akande@uga.edu}

\address{Department of Mathematics\newline
University of Georgia\newline
Athens, Georgia 30602, U.S.A.}
\email{robert.schneider@uga.edu}

\title{Semi-modular forms from Fibonacci--Eisenstein series }
 
\begin{document}

\begin{abstract}
In recent work, M. Just and the second author defined a class of ``semi-modular forms'' on $\mathbb C$, in analogy with classical modular forms, that are ``half modular'' in a particular sense; and constructed families of such functions as Eisenstein-like series using symmetries related to integer partitions. Looking for further natural examples of semi-modular behavior, here we construct a family of Eisenstein-like series to produce semi-modular forms, using symmetries related to Fibonacci numbers instead of partitions. We then consider other Lucas sequences that yield semi-modular forms. 
\end{abstract}
\maketitle

\section{Introduction and statement of results}
  
In a recent paper \cite{JS}, M. Just and the second author defined a class of ``semi-modular forms'' on $\mathbb C$, in analogy with classical modular forms, that are ``half modular'' in a particular sense; and produced examples of such functions using special constructions related to integer partitions that give a combinatorial interpretation to the transformation $z \mapsto -1/z$.\footnote{Inversion $z\mapsto -1/z$ is related to conjugation of Ferrers-Young diagrams of partitions in \cite{JS}.} Quite apart from the theory of partitions, here we construct new examples of semi-modular forms using symmetries from Fibonacci numbers and Lucas sequences. 

We give a quick overview of modular forms theory in order to describe semi-modularity.  Let us recall the canonical generators of the {\it general linear group} $GL_{2}(\mathbb Z)$ are 
\begin{equation*}
T=\begin{pmatrix}
1 & 1 \\
0 & 1 
\end{pmatrix},
\  \  \  \  
U=\begin{pmatrix}
0 & 1 \\
1 & 0 
\end{pmatrix},
\  \  \  \  
V=\begin{pmatrix}
1 & 0 \\
0 & -1 
\end{pmatrix},
\end{equation*}
such that $GL_{2}(\mathbb Z)=\left<T,U,V\right>$. An important subgroup of $GL_{2}(\mathbb Z)$ is the {\it modular group} $PSL_{2}(\mathbb Z)$, which is well known to be generated by $T$ together with the matrix 
\begin{equation*}
S=\begin{pmatrix}
0 & -1 \\
1 & 0 
\end{pmatrix}
\in GL_{2}(\mathbb Z).\end{equation*}
Functions of $z$ in the upper half-plane $\mathbb H$ invariant under $\left<S,T\right>=PSL_{2}(\mathbb Z)$ up to a simple multiplier in $z$ are {\it modular forms}, a class of functions central to number theory \cite{Apostol, Ono_web}.

For a canonical example, recall the weight-$m$ {\it Eisenstein series}, the prototype of an integer weight holomorphic modular form, convergent for $m>2, z\in \mathbb H$:
\begin{equation}\label{Eisenstein}
G_{m}(z)=\sum_{\substack{j,k\in\mathbb Z \\ (j,k)\neq (0,0)}}  {(jz+k)^{-m}}.
\end{equation}
Note that if $m$ is odd, every term cancels its negative, and $G_m(z)=0$. Taking $m \mapsto 2m$, for $m>1$ the function $G_{2m}:\mathbb H \to \mathbb C$ satisfies the defining properties of a modular form: 
\begin{enumerate}[(i)]
\item $G_{2m}(-\frac{1}{z})=z^{2m}G_{2m}(z)$\   (weighted invariance under inversion matrix $S$),
\item $G_{2m}(z+1)=G_{2m}(z)$ \  (invariance under translation matrix $T$).
\end{enumerate}

    The idea of a {\it semi-modular} form is essentially that it is a function of a complex variable invariant under {\it one} of the two matrices $S,T$, together with being invariant under a second matrix $M$ that is complementary in a natural sense to $S,T$ -- thus, as we mentioned above, it is ``half modular''. 
    To give more detail about the sense in which matrix $M$ complements $PSL_2(\mathbb Z)$, we want $M\in GL_2(\mathbb Z)$ to be a ``nice'' transformation matrix  such that one can partition the general linear group, in terms of its generators, into the modular group $PSL_2(\mathbb Z)$ together with $M$; i.e., such that one can write $$GL_2(\mathbb Z)=\left<M,S,T\right>.$$ In principle, the matrices $M,S,T$ yield three families of functions invariant on $\left<S,T\right>$, $\left<M,S\right>$ and $\left<M,T\right>$, respectively.  We count modular forms as semi-modular as well, but the existence of a complementary matrix $M$ is implicit in the definition. 
    
    We note that finding a matrix $M$ to complement $S,T$ is all but trivial; the details are in identifying matrices $M$ that are canonical in some sense, and that also admit reasonably natural functions invariant on $\left<M,S\right>$ and $\left<M,T\right>$. In \cite{JS}, the ``even function'' matrix 
\begin{equation}\label{even}
M'=\begin{pmatrix}
-1 & 0 \\
0 & 1 
\end{pmatrix}
\end{equation}
was  used to produce such families: the family $\left<M',T\right>$ includes functions like $\cos z, z\in \mathbb C$. Using detailed constructions in partition theory, 
the authors define ``partition Eisenstein series'' invariant on $\left<M',S\right>$, that are analogous to \eqref{Eisenstein} but summed over integer partitions. 

Now, partitions and modular forms are closely connected classically, so it seems natural to look to partitions for generalizations of modularity.\footnote{See e.g. papers on the $q$-bracket of Bloch and Okounkov such as \cite{BO, Ono_bracket, Schneider_JTP, Jan_bracket, Zagier_bracket}.} In \cite{JS}, which uses symmetries in the Ferrers-Young diagrams of partitions in a very specific way to induce semi-modular behavior, the authors ask if other classes of semi-modular forms 
arise from natural symmetries in areas of mathematics beyond the confluence of partitions and modular forms. 

We will now show that semi-modularity does arise outside of the usual partitions-modular forms universe; we construct a family of semi-modular forms using symmetries related to classical Fibonacci numbers instead of partitions. 
Recall for $n\geq 0$, the $n$th {\it Fibonacci number} $F_n$ is defined by $F_0:=0,\  F_1:=1$, and for $n>1$,
\begin{equation}\label{Fibdef}
F_n\  :=\  F_{n-1}+F_{n-2}.
\end{equation}
This recursion extends to negative indices by defining, for $n\geq 1$,
$
F_{-n}\  :=\  (-1)^{n-1} F_n.
$
 
Much like partitions, Fibonacci numbers are additive objects that appear throughout the mathematical sciences, particularly due to their close connection to the {\it golden ratio} $\phi:= \frac{1+\sqrt{5}}{2}=\lim_{n\to\infty}F_{n}/F_{n-1}$. Via evaluations of the Rogers-Ramanujan continued fraction \cite{Sills}, the golden ratio plays a role in the theory of modular forms -- giving a second-hand link to modular forms for Fibonacci numbers. More recent works directly connect Fibonacci numbers and modular forms (see e.g. \cite{Murty, Fibmat}). 

Seeking new examples of semi-modular forms, 
we will define a Fibonacci variant of \eqref{Eisenstein}. 

\begin{definition}\label{F} For fixed $m>1$, $z\in \mathbb C, z \neq \phi, -\frac{1}{\phi}$ or  $\frac{F_n}{F_{n-1}}\  \text{for any}\  n\in \mathbb Z$, 
let 
\begin{flalign*} 
{\mathscr F}_{m}(z):=\sum_{j=-\infty}^{\infty} {(F_j z + F_{j-1})^{-m}},\end{flalign*}
where the $F_n$ are Fibonacci numbers, $n \in \mathbb Z$. We call $m>1$ the ``weight'' of the series.
\end{definition}

We note this is a single bilateral summation, as opposed to a double sum as in \eqref{Eisenstein}, and is convergent on its domain when $m>1$ by comparison with $\zeta(m)$. To justify the domain conditions, note that for any $n\in \mathbb Z, n\neq 1$, the term $F_{-n+1}z+F_{-n}$ vanishes in the denominator when $z=F_{n}/F_{n-1}$ due to the minus sign attached to exactly one of the negatively-indexed terms, giving a pole of order $m$. Since $z=\phi, -\frac{1}{\phi}$ are the respective limits of the infinite sequence of poles $F_{n}/F_{n-1}$ and $F_{-n}/F_{-n-1}$ as $n\to \infty$, these limiting values themselves represent essential singularities.
 
 This ``Fibonacci-Eisenstein'' series ${\mathscr F}_{k}$ has nice transformation properties at even weights.

\begin{theorem}\label{thm2}
Let $z\in \mathbb C, z \neq \phi, -\frac{1}{\phi}, \frac{F_{n}}{F_{n-1}}\  \text{for any}\  n\in \mathbb Z$. 
For $k\geq 1$ we have:
\begin{enumerate}[(i)]
\item $\mathscr F_{2k}(-\frac{1}{z})=z^{2k}\mathscr F_{2k}(z)$; 
\item $\mathscr F_{2k}(1-z)=\mathscr F_{2k}(z)$. 
\end{enumerate}
\end{theorem}
 
\begin{remark} We note that in these functions, the odd-weight cases do not vanish like \eqref{Eisenstein}.
\end{remark}
The property (i) is, of course, weighted invariance with respect to matrix $S$. Property (ii), symmetry around $\operatorname{Re}(z)=1/2$, is encoded in the matrix
\begin{equation*}
P=\begin{pmatrix}
-1 & 1 \\
0 & 1 
\end{pmatrix}.
\end{equation*}
The Riemann zeta function $\zeta(z)$ in the critical strip is an example of a function with (weighted) invariance around $P$, as are the trigonometric functions $
\sin(\pi z),\ \cos(2 \pi z)
$. 

Now, noting that, in fact, the generators of $GL_2(\mathbb Z)$ can be written \begin{equation}\label{gen}
U=PTS,\  \  \  \  V=SPTS^{3}, 
\end{equation}
then one can alternatively view the general linear group as
\begin{equation*}
GL_2(\mathbb Z)=\left< P,S,T\right>.\end{equation*}
From this perspective, modular forms invariant on $\left<S,T\right>$, and periodic functions of period 1 symmetric around the $1/2$-line invariant on $\left<P,T\right>$, are members of a larger class of {semi-modular forms} invariant on two of the three generators $P,S,T$ of $GL_2(\mathbb Z)$. By Theorem \ref{thm2},  $\mathscr F_{2k}(z)$ is a semi-modular form in this class, too, invariant on $\left<P,S\right>$.

\begin{remark}
We note $PS$ is a so-named Fibonacci matrix (see e.g.  \cite{Fibmat}) such that, for $n\geq1$,
\begin{equation*}
(PS)^n=\begin{pmatrix}
1 & 1 \\
1 & 0 
\end{pmatrix}^n
=\begin{pmatrix}
F_{n+1} & F_n \\
F_n & F_{n-1} 
\end{pmatrix}.
\end{equation*}
\end{remark}

In Section \ref{sect2}, we will establish the semi-modularity of $\mathscr F_{2k}$. 
In Section \ref{sect3} we discuss connections to other Lucas sequences, and produce infinite families of semi-modular forms.

\section{Proof of the main result}\label{sect2}
Our proof of Theorem \ref{thm2} depends on the following decomposition of the bilateral series $\mathscr{F}_{m}(z)$ into two unilateral parts:
\begin{equation}\label{decomp}
\mathscr{F}_{m}(z)\  =\  \mathscr{F}_{m}^-(z)\  +\  \mathscr{F}_{m}^+(z),\end{equation}
where 
\begin{equation*}\mathscr{F}_{m}^-(z) := \sum \limits_{-\infty<n\leq 0} {(F_{n}z + F_{n-1})^{-m}},\  \  \  \  \  \  \mathscr{F}_{m}^+(z) := \sum \limits_{1\leq n<\infty}  {(F_{n}z + F_{n-1})^{-m}}.\end{equation*}
\begin{proof}[Proof of Theorem \ref{thm2}] Recall $F_0=0$ and for negative indices we define $F_{-n}=(-1)^{n-1}F_n$. We only use even values $m=2k$ with $k>1$ (even powers regularize the $\pm$ sign behavior in our proof below). We will employ standard Eisenstein series techniques, together with the recursion \eqref{Fibdef}. Note for the positively indexed terms of $\mathscr{F}_{2k}$ we have 
\begin{flalign}\label{eq2}\mathscr{F}_{2k}^+ \left(z+1\right) &= \sum \limits_{n\geq 1} {\left(F_n\left(z+1\right) + F_{n-1}\right)^{-2k}} =  \sum \limits_{n\geq 1}   {\left(F_{n}z + (F_{n} + F_{n-1})\right)^{-2k}}\\ \nonumber
&= \sum \limits_{n\geq 1}  {(F_{n}z + F_{n+1})^{-2k} } = z^{-2k}\sum \limits_{n\geq 2}  {\left(F_{n}(1/z) + F_{n-1}\right)^{-2k} } \\ \nonumber &= z^{-2k}\mathscr{F}_k^+\left(1/z\right)-1.\end{flalign}
Similarly, for the non-positively indexed terms of $\mathscr{F}_{2k}$ we have
\begin{flalign}\label{eq1}\mathscr{F}_{2k}^- \left(z+1\right) &= \sum \limits_{n\leq 0}  {\left(F_n\left(z+1\right) + F_{n-1}\right)^{-2k}} =  \sum \limits_{n\leq 0}  {(F_{n}z + (F_{n} + F_{n-1}))^{-2k}}\\ \nonumber
&= \sum \limits_{n\leq 0} {\left(F_{n}z + F_{n+1}\right)^{-2k} } = z^{-2k}\sum \limits_{n\leq 1}  {\left(F_{n}(1/z) + F_{n-1}\right)^{-2k} }\\ \nonumber  &= z^{-2k}\mathscr{F}_{2k}^-\left(1/z\right)+1.\end{flalign}
Then using \eqref{decomp} to add the corresponding left- and right-hand sides of \eqref{eq2} and \eqref{eq1} yields \begin{equation}\label{eq2.5}\mathscr{F}_{2k}  \left(z+1\right) = z^{-2k}\mathscr{F}_{2k}\left(1/z\right).\end{equation}
Along similar lines, we also have 
\begin{flalign}\label{eq3}\mathscr{F}_{2k}^- \left(-z\right) &= \sum \limits_{n\leq 0}  {\left(F_n\left(-z \right) + F_{n-1}\right)^{-2k}} = z^{-2k}\sum \limits_{n\leq 0} {\left(-F_n + F_{n-1}(1/z)\right)^{-2k}}\\ \nonumber
&= z^{-2k}\sum \limits_{n\geq 1} {\left(F_{n}(1/z)+F_{n-1}\right)^{-2k}} =z^{-2k}\mathscr{F}_{2k}^+\left(1/z\right),\end{flalign}
as well as
\begin{flalign}\label{eq4}\mathscr{F}_{2k}^+ \left(-z\right) &= \sum \limits_{n\geq 1}{\left(F_n\left(-z \right) + F_{n-1}\right)^{-2k}} = z^{-2k}\sum \limits_{n\geq 1} {\left(-F_n + F_{n-1}(1/z)\right)^{-2k}}\\ \nonumber
&= z^{-2k}\sum \limits_{n\leq 0}{\left(F_{n}(1/z)+F_{n-1}\right)^{-2k}} =z^{-2k}\mathscr{F}_{2k}^-\left(1/z\right).\end{flalign}
We note here that the terms of $\mathscr{F}_{2k}^-$ and $\mathscr{F}_{2k}^+$ swap places to produce the inversion $-z\mapsto 1/z$. As above, adding the corresponding left- and right-hand sides of \eqref{eq3} and \eqref{eq4} yields
\begin{flalign}\label{eq4.5}\mathscr{F}_{2k}  \left(-z\right) = z^{-2k}\mathscr{F}_{2k}\left(1/z\right).\end{flalign}
Making the substitution $z\mapsto -z$, then comparing equations \eqref{eq2.5} and \eqref{eq4.5}, a little algebra gives the theorem. The domain restrictions on $z$ were justified below Definition \ref{F}.\end{proof}

\section{Extension to other Lucas sequences}\label{sect3}

At this stage, seeing how the proofs above follow easily from the recursion \eqref{Fibdef}, a natural next question is: are other Fibonacci-like sequences such as Lucas sequences subject to the same treatment, leading to further families of semi-modular functions? Recall for $a,b \in \mathbb Z$ that a classical {\it Lucas sequence}\footnote{For a comprehensive survey of Fibonacci and Lucas sequences, see \cite{Fib}.} $\{L_n(a,b)\}$ is defined for $n\geq 2$ by the recursion
\begin{equation}\label{Lucasdef}
L_n(a,b)\  :=\  a L_{n-1}(a,b)-b L_{n-2}(a,b).
\end{equation}
With the initial conditions $L_0(a,b)=0,\  L_1(a,b)=1,$ it is called a {\it Lucas sequence of the first kind.} When $L_0(a,b)=2,\  L_1(a,b)=a,$ it is called a {\it Lucas sequence of the second kind}. Lucas sequences generalize a number of classical Fibonacci-like sequences; for instance, the Lucas sequence $L_n(1,-1)$ of the first kind is the Fibonacci sequence $F_n$, and the same form $L_n(1,-1)$ of the second kind defines the classical Lucas numbers $L_n$.

However, for our proof above to extend to other sequences, 
we require more than just the recursion, an even weight, and standard Eisenstein series manipulations. We need the initial terms defined such that $L_{n}(a,b)$ has ``alternating sign symmetry''  $L_{-n}(a,b)=(-1)^{n-1} L_n(a,b)$ around the term $L_0(a,b)$ (or a similar near symmetry), so the positively and negatively indexed terms will make a swap as noted below equations \eqref{eq3} and \eqref{eq4}. 

The Lucas numbers $L_n=L_n(1,-1)$ with initial conditions $L_0:=2, L_1:=1,$ and recursion $L_n:=L_{n-1}+L_{n-2}$, admit the similar extension $L_{-n}=(-1)^nL_n$. 

\begin{theorem}\label{thm3}
Let $m>1,   z\in \mathbb C, z \neq \phi, -\frac{1}{\phi}, \frac{L_{n}}{L_{n-1}}\  \text{for any}\  n\in \mathbb Z$, and for $m>1$ define
\begin{flalign*} 
{\mathscr L}_{m}(z):=\sum_{j=-\infty}^{\infty}{(L_j z + L_{j-1})^{-m}},\end{flalign*} 
where the $L_n$ are Lucas numbers, $n\in \mathbb Z$. Then for $k\geq 1$ we have:
\begin{enumerate}[(i)]
\item $\mathscr L_{2k}(-\frac{1}{z})=z^{2k}\mathscr L_{2k}(z)$; 
\item $\mathscr L_{2k}(1-z)=\mathscr L_{2k}(z)$. 
\end{enumerate}
\end{theorem}

\begin{proof}
Replace $F_n$ with $L_n$ in the proof of Theorem \ref{thm2}, noting the arguments about the poles and singularities at $\phi, -1/\phi$ also hold here since $\lim_{n\to \infty}L_n/L_{n-1}=\phi$.
\end{proof}

Because  the function $\mathscr L_{2k}(z)$ enjoys (weighted) invariance with respect to $\left<P,S\right>$ just like $\mathscr F_{2k}(z)$, it is a semi-modular form in the same class.

The ``alternating sign symmetry'' required for these proofs is not a general property of Lucas sequences, but can be obtained in more general cases. 

\begin{definition}\label{L} For fixed $a,b\neq 0, m>1, z\in \mathbb C, z\neq \frac{L_{n}(a,b)}{L_{n-1}(a,b)}$ for any $n\in \mathbb Z, z \neq \lim_{n\to \infty}\frac{L_{n}(a,b)}{L_{n-1}(a,b)}$ or $\lim_{n\to \infty}\frac{-L_{n-1}(a,b)}{L_{n}(a,b)}$ if the limits exist,  
let 
\begin{flalign*} 
{\mathscr L}_{a,b,m}(z):=\sum_{j=-\infty}^{\infty}{\left(L_j(a,b)z + L_{j-1}(a,b)\right)^{-m}},\end{flalign*}
where $L_n(a,b)$ is a Lucas sequence of either the first or the second kind (to be specified).
\end{definition}

The analytic conditions generalize those in Definition \ref{F} with identical justifications, noting by \eqref{Lucasdef} that ${\mathscr L}_{0,b,m}(z)$ clearly diverges by consideration of geometric series.

Taking $b=-1$ in Lucas sequences $L_n(a,b)$ of the {\it first} kind, one sets $L_0(a,-1):=0,\  L_1(a,-1):=1,$ and for $a \in \mathbb Z, a\neq 0$, one has
$
L_n(a,-1)\  :=\  a L_{n-1}(a,-1)+L_{n-2}(a,-1).
$ 
These $L_n(a,-1)$ are referred to as {\it $a$-Fibonacci numbers}, and generalize the properties of $F_n$. 
%
It is easy to check that we can extend the indices using $L_{-n}(a,-1):=(-1)^{n-1} L_n(a,-1)$ for this Lucas sequence of the first kind. Moreover, a similar relation $L_{-n}(a,-1):=(-1)^{n} L_n(a,-1)$ also holds when $a\neq 0$ if it is taken as a Lucas sequence of the {\it second} kind, setting $L_0(a,-1):=2,\  L_1(a,-1):=a$. 

%

%

\begin{theorem}\label{thm4}
For $a\in \mathbb Z\backslash \{0\}$, let $L_n(a,-1)$ be a Lucas sequence of either the first or the second kind.  Take $z\in \mathbb C, z\neq \frac{L_{n}(a,b)}{L_{n-1}(a,b)}$ for any $n\in \mathbb Z, z \neq \lim_{n\to \infty}\frac{L_{n}(a,-1)}{L_{n-1}(a,-1)}$ or $\lim_{n\to \infty}\frac{-L_{n-1}(a,-1)}{L_{n}(a,-1)}$. 
Then for $k\geq 1$ we have:
\begin{enumerate}[(i)]
\item $\mathscr L_{a,-1,2k}(-\frac{1}{z})=z^{2k}\mathscr L_{a,-1,2k}(z)$; 
\item $\mathscr L_{a,-1,2k}(a-z)=\mathscr L_{a,-1,2k}(z)$. 
\end{enumerate}
\end{theorem}
 \begin{remark}
 Setting $a=1$, then Theorem \ref{thm4} subsumes both Theorem \ref{thm2} and Theorem \ref{thm3}.
 \end{remark}

\begin{proof} This is very similar to the proof of Theorem \ref{thm2}; one could decompose $L_n(a,-1)$  into positively  and non-positively indexed terms and carry out the same manipulations. Here we summarize those steps without decomposing into two halves. First of all, we have 
\begin{flalign}\label{eqL1}  
{\mathscr L}_{a,-1,2k}(z+a)&=\sum_{j=-\infty}^{\infty} {\left[L_j(a,-1)(z+a) + L_{j-1}(a,-1)\right]^{-2k}}\\ 
\nonumber &= \sum_{j=-\infty}^{\infty} {\left[L_j(a,-1)z+ (aL_j(a,-1)+L_{j-1}(a,-1))\right]^{-2k}}\\
\nonumber &=z^{-2k}\sum_{j=\infty}^{-\infty} {\left[-L_j(a,-1)+ L_{j+1}(a,-1)(1/z)\right]^{-2k}}=z^{-2k}{\mathscr L}_{a,-1,2k}(1/z),
\end{flalign}
where we reverse indices $\sum_{n=-\infty}^{\infty} \mapsto \sum_{n=\infty}^{-\infty}$ in the final summation to indicate the positively and non-positively indexed terms make a swap just as we noted in equations \eqref{eq3} and \eqref{eq4}. Similarly, we have
\begin{flalign}\label{eqL2}  
{\mathscr L}_{a,-1,2k}(-z)&=\sum_{j=-\infty}^{\infty} {\left[L_j(a,-1)(-z) + L_{j-1}(a,-1)\right]^{-2k}}\\ 
\nonumber &=z^{-2k}\sum_{j=\infty}^{-\infty}{\left[-L_j(a,-1)+ L_{j-1}(a,-1)(1/z)\right]^{-2k}}=z^{-2k}{\mathscr L}_{a,-1,2k}(1/z).
\end{flalign}
Taking $z\mapsto -z$ and comparing \eqref{eqL1} and \eqref{eqL2}, with a little algebra, completes the proof. 
\end{proof}
Then ${\mathscr L}_{a,-1,2k}(z)$ is invariant with respect to $\left<P_a, S\right>$, where $P_a$ is the matrix
\begin{equation*}
P_a=\begin{pmatrix}
-1 & a \\
0 & 1 
\end{pmatrix}
\end{equation*}
representing mirror symmetry around $\operatorname{Re}(z)=a/2$.\footnote{The matrix $P_a$ plays a role in work in the literature relating Fibonacci matrices to modular forms and Poincar\'{e} series; and $P_3S$ is noted in \cite{Fibmat} to be ``almost a Fibonacci matrix''. } 
Noting that $P_a T^a=PT$, then by \eqref{gen} we also have  
\begin{equation*}
U=P_aT^aS,\  \  \  \  V=SP_aT^aS^{3}. 
\end{equation*}
Thus one can write $$GL_2(\mathbb Z)=\left< P_a,S,T\right>,$$ and ${\mathscr L}_{a,-1,2k}(z)$ satisfies our definition of a semi-modular form.


It follows from the recursion \eqref{Lucasdef} that, in general, one can define
\begin{equation}\label{alt} L_{-n}(a,b) := (-1)^{\ell} L_{n}(a,b) / b^n, \end{equation}
where $\ell=1$ if $L_{n}(a,b)$ is a Lucas sequence of the first kind, and $\ell=2$ if it is a Lucas sequence of the second kind. From \eqref{alt} it is clear $b=-1$ is the only substitution that produces the desired ``alternating sign symmetry'' for sequences of either kind. 
Perhaps other substitutions, like setting $b$ equal to other roots of unity, would yield further relations for ${\mathscr L}_{a,b,2k}(z)$.\footnote{One can produce equivalent forms as slight variants on these series -- e.g., 
$\sum_{j=-\infty}^{\infty} {(F_j-F_{j-1}z)^{-2k}}$ 
is semi-modular on $\left<P,S\right>$ just like $\mathscr F_{2k}(z)$ -- so other nice variants of  ${\mathscr L}_{a,b,2k}(z)$ seem possible.}

Seeing how the $a=0$ case of $P_a$ is the complementary matrix \eqref{even} used in \cite{JS}, then the set $\left\{P_a   :  \left< P_a,S,T\right>=GL_2(\mathbb Z), a\in \mathbb Z\right\}$ of mirror symmetries around $\operatorname{Re}(z)=a/2$ is the complete collection of complementary matrices that have so far produced examples of semi-modular forms. Do there exist matrices that are {\it not} of the form $P_a$, that both complement $S,T$ as generators of $GL_2(\mathbb Z)$ and produce natural semi-modular forms? Based on connections such as those proved in \cite{Fibmat} between $P_a$, Fibonacci matrices and Poinc\'{a}re series, are there further links between semi-modularity and classical modular forms? Moreover, since we now find semi-modularity to arise from symmetries in Lucas sequences, as well as from partition theory as in \cite{JS}, one wonders: what other symmetric, additive or recursive structures in mathematics might be used to construct semi-modular Eisenstein-like series? 

\section*{Acknowledgments}
We are thankful to Marie Jameson, Matthew Just, Maxwell Schneider and A. V. Sills for conversations that informed this study.

\end{document}